\documentclass[10pt]{amsart}

\usepackage{amssymb,amsmath,amsthm,amscd}

\advance\oddsidemargin by -1cm
\advance\evensidemargin by -1cm 
\newtheorem{theorem}{Theorem}

\newtheorem{proposition}{Proposition}
\newtheorem{corollary}{Corollary}
\newtheorem{lemma}{Lemma}
\newtheorem{definition}{Definition}

\theoremstyle{definition}
\newtheorem{remark}{Remark}

\newcommand{\bdm}{\begin{displaymath}}
\newcommand{\edm}{\end{displaymath}}
\newcommand{\bq}{\begin{equation}}
\newcommand{\eq}{\end{equation}}
\newcommand{\bqn}{\begin{equation*}}
\newcommand{\eqn}{\end{equation*}}

\newcommand{\rn}{\mathbb{R}^n}

\newcommand{\mklm}[1]{\left\{ #1 \right\}}
\newcommand{\eklm}[1]{\left\langle  #1 \right\rangle}

\newcommand{\N}{{\mathbb N}}

\newcommand{\C}{{\mathbb C}}
\newcommand{\R}{{\mathbb R}}

\newcommand{\D}{{\mathcal D}}
\newcommand{\E}{{\mathcal E}}

\newcommand{\X}{{\mathbb X}}

\renewcommand{\epsilon}{\varepsilon}
\renewcommand{\phi}{\varphi}
\renewcommand{\rho}{\varrho}
\newcommand{\1}{{ \bf  1}}

\newcommand{\Cinft}{{\rm C^{\infty}}}
\newcommand{\CT}{{\rm C^{\infty}_c}}

\renewcommand{\L}{{\rm L}}
\newcommand{\Lcal}{{\mathcal L}}

\renewcommand{\S}{{\mathcal S}}

\newcommand{\Sym}{{\rm S}}
\newcommand{\Syms}{{\rm S^{-\infty}}}
\newcommand{\Symsl}{{\rm S^{-\infty}_{la}}}

\newcommand{\SL}{\mathrm{SL}}
\newcommand{\SO}{\mathrm{SO}}

\newcommand{\g}{{\bf \mathfrak g}}
\renewcommand{\k}{{\bf \mathfrak k}}
\renewcommand{\a}{{\bf\mathfrak a}}

\newcommand{\n}{{\bf\mathfrak n}}
\newcommand{\p}{{\bf \mathfrak p}}
\newcommand{\h}{{\bf \mathfrak h}}

\newcommand{\U}{{\mathfrak U}}

\newcommand{\spl}{{\bf \mathfrak{sl}}}

\newcommand{\Ad}{\mathrm{Ad}\,}
\newcommand{\ad}{\mathrm{ad}\,}
\newcommand{\sgn }{\mathrm{sgn }\,}

\renewcommand{\det}{\mathrm{det}\,}

\renewcommand{\Re}{\mathrm{Re}\,}

\newcommand{\Fix}{\mathrm{Fix}}

\DeclareMathOperator{\supp}{supp}

\DeclareMathOperator{\rank}{rank}
\DeclareMathOperator{\Tr}{Tr}
\DeclareMathOperator{\gd}{\partial}

\newcommand{\dbar}{{\,\raisebox{-.1ex}{\={}}\!\!\!\!d}}

\begin{document}

\author{Aprameyan Parthasarathy and Pablo Ramacher}
\title[Integral operators on Oshima compactifications of
Riemannian symmetric spaces]{Integral operators on the
Oshima compactification of a Riemannian symmetric space
of non-compact type. Regularized traces and characters}
\address{Aprameyan Parthasarathy and Pablo Ramacher,
Fachbereich Mathematik und Informatik,
Philipps-Universit\"at Marburg,
Hans-Meerwein-Str., 35032 Marburg, Germany}
\email{apra@mathematik.uni-marburg.de,
ramacher@mathematik.uni-marburg.de}
\thanks{This work was financed by the DFG-grant RA
1370/2-1.}
\date{June 2, 2011}

\begin{abstract} Consider a Riemannian symmetric space  $\X= G/K$ of
non-compact type, where $G$ denotes a connected, real, semi-simple Lie group with finite center, and $K$ a maximal compact subgroup of $G$. Let $\widetilde \X$ be its Oshima compactification, and $(\pi,\mathrm{C}(\widetilde \X))$ the regular
representation of $G$ on $\widetilde \X$. 
In this paper, a regularized trace  for the convolution operators $\pi(f)$ is defined, yielding a distribution on $G$ which can be interpreted as   global character of $\pi$. In case that $f$ has  compact support in a certain set of transversal elements, this distribution is a locally integrable function, and given by a fixed point formula analogous to the formula for the global character of an induced representation of $G$. 
\end{abstract}

\maketitle

\tableofcontents

\section{Introduction}

Let $G$ be a connected, real,  semi-simple Lie group with finite center,  $K$ a maximal compact subgroup, and $G/K$ the corresponding Riemannian symmetric space which is assumed to be of non-compact type.
 In this paper, a distribution character for the regular representation of $G$ on the Oshima compactification of $G/K$  is introduced, and a corresponding character formula is proved. The paper is a continuation of \cite{parthasarathy-ramacher11}, to which we shall refer in the following as Part I. 

In his early work on infinite dimensional representations of semi-simple Lie groups, Harish--Chandra \cite{harish-chandra54} realized that the correct generalization of the character of a finite-dimensional representation was a distribution on the group given by the trace of a convolution operator on representation space. This distribution  character is given by a locally integrable function which is analytic on the set of regular elements, and  satisfies character formulas analogous to the finite dimensional case. Later, Atiyah and Bott  \cite{atiyah-bott68} gave a similar description of the  character of a parabolically  induced representation  in their work on  Lefschetz fixed point formulae for elliptic complexes. More precisely, let  $H$ be  a closed co-compact subgroup of $G$,  and $\rho$ a representation of $H$ on a finite dimensional vector space $V$. If $T(g)=(\iota_\ast \rho)(g)$ is the representation of $G$ induced by $\rho$ in the space of sections over $G/H$ with values in the homogeneous vector bundle $G\times_HV$, then its  distribution character   is  given by the distribution 
\bqn 
\Theta_T:\CT(G) \ni f \longmapsto \Tr \, T(f), \qquad T(f)=\int _G f(g) T(g) d_G(g),
\eqn
where  $d_G$ denotes a Haar measure on $G$.  The point to be noted is that $T(f)$ is a smooth operator, and since $G/H$ is compact, does have a well-defined trace.  On the other hand, assume that    $g\in G$ acts on $G/H$ only with simple fixed points. In this case, a transversal trace     $\Tr^\flat T(g)$ of  $T(g) $ can be defined within the framework of pseudodifferential operators, which is given by a sum over fixed points of $g$.  Atiyah and Bott then showed that, on an open  set $G_T\subset G$,
\begin{equation*}
 \Theta_T(f)=\int_{G_T} f(g) \Tr^\flat T(g) d_G(g), \qquad f \in \CT(G_T).
\end{equation*}
This means that, on $G_T$, the character $\Theta_T$ of the induced representation $T$ is  represented by the locally integrable function $\Tr^\flat T(g)$, and its computation   reduced to  the evaluation of a sum over fixed points. When $G$ is a p-adic reductive group defined over a non-Archimdean local field of characteristic zero, a similar analysis of the character of a parabolically induced representation was  carried out in \cite{clozel}. 

In this paper, we consider the regular representation $\pi$ of $G$ on the Oshima compactification $\widetilde \X$ of a Riemannian symmetric space $\X= G/K$ of non--compact type given by
\bqn 
\pi(g) \phi(\tilde x) = \phi(g^{-1} \cdot \tilde x), \qquad \phi \in \Cinft(\widetilde \X).  
\eqn
 Since the $G$-action on $\widetilde \X$ is not transitive, 
the corresponding convolution operators $\pi(f)$,   $ f \in \CT(G)$, are not smooth, and therefore do not have a well-defined trace. Nevertheless, it was shown in Part I that they can be characterized as totally characteristic pseudodifferential operators of order $-\infty$. Using this fact, we are able to define  a regularized trace $\Tr_{reg} \pi(f)$ for the operators $\pi(f)$, and in this way obtain a map 
$$
\Theta_\pi:\CT(G) \ni f \mapsto \Tr_{reg}(f) \in \C,
$$ 
 which is shown to be a distribution on $G$. This distribution is defined to be the character of the representation $\pi$. We then  show that,  on a certain open set $G(\widetilde \X)$ of transversal elements,
\bqn
\Tr_{reg} \pi(f)=\int_{G(\widetilde \X)} f(g) \Tr^\flat \pi(g) d_G(g), \qquad f \in \CT(G(\widetilde \X)),
\eqn
where, with the notation $\Phi_g(\tilde x)=g \cdot \tilde x$,
\bqn
\Tr^\flat \pi(g)=  \sum_{\tilde x \in \mathrm{Fix}(g)}     \frac{1 }{ |\det (1 -d\Phi_{g^{-1}}(\tilde x))|},
\eqn
the sum being over the (simple) fixed points of $g \in G(\widetilde \X)$ on $\widetilde \X$. Thus,  on the open set  $G(\widetilde \X)$, $\Theta_\pi$ is  represented by the locally integrable function $\Tr^\flat \pi(g)$,  which is given by a formula similar to the character of a parabolically induced representation.  It is likely that similar distribution characters could be introduced for  $G$-manifolds with a dense union of open orbits, or for spherical varieties, and that corresponding character formulae could be proved.

This paper is structured as follows. In Section \ref{sec:prelim}, the main results of Part I  that will be needed in the sequel are recalled. The regularized trace for the convolution operators $\pi(f)$ is defined in Section \ref{Sec:3}, while  the transversal trace of a pseudodifferential operator is introduced in Section \ref{sec:AB}, followed by a discussion of the global character of an induced representation. After studying $G$-actions on homogeneous spaces in Section \ref{Sec:4}, we  prove  that  the distribution $\Theta_\pi$ is regular on the set of transversal elements  $G(\widetilde \X)$, and given by the locally integrable function $\Tr^\flat \pi(g)$. This is done in Section \ref{Sec:6}. In the last section, the Oshima compactification of  $\X=\SL(3,\R)/\SO(3)$ is described in detail.

\section{Preliminaries}
\label{sec:prelim}

In this section we shall briefly recall the main results of  Part I relevant to our purposes.  Let $G$ be a connected, real, semi-simple Lie group with
finite centre and Lie algebra $\g$, and denote by
$\langle X,Y\rangle = \Tr \, (\ad X\circ \ad Y)$ the
\emph{Cartan-Killing form} on $\g$. Let $\theta$ be a
Cartan involution on $\g$, and
$$\g = \k\oplus\p$$
the corresponding Cartan decomposition. Put $\langle X,Y\rangle
_\theta:=-\langle X,\theta Y\rangle $. Consider further a maximal
Abelian subspace $\a$ of $\p$. Then $\ad(\a)$ constitutes a commuting
family of self-adjoint operators on $\g$ relative to
$\langle,\rangle_{\theta}$, and one defines for each
$\alpha\in\a^*$  the simultaneous
eigenspaces $\g^\alpha=\{X\in\g:[H,X]=\alpha(H)X \,
\text{for all } H\in\a\}$. Let 
$\Sigma=\{\alpha\in\a^*:\alpha\neq0,\g^\alpha\neq\{0\}\}$ be the set of roots of $(\g,\a)$,
 $\Sigma^{+}=\{\alpha\in\Sigma:\alpha> 0\}$  a \emph{set of positive roots}, and
$\Delta=\{\alpha_1,\dots ,\alpha_l\}$ the \emph{set of simple roots}. 
Define $\n^+=\bigoplus_{\alpha\in\Sigma^{+}}\g^\alpha,
\, \n^-=\theta(\n^+)$, and write $K, A, N^+$ and $N^- $
for the analytic subgroups of $G$ corresponding to
$\k,\, \a,\, \n^+$, and $ \n^-$, respectively. Let $M$ and $M^\ast$ be the centralizer and the normalizer of $ \a $ in $K$, respectively. 
Consider then the Oshima compactification $\widetilde{\X}$ of the Riemannian symmetric space $\X=G/K$  which is assumed to be of non-compact type.  It is a simply connected, compact,
real-analytic manifold without boundary carrying a real-analytic $G$-action. The corresponding orbital
decomposition of $\widetilde{\X}$ is of the form
\bq
\label{eq:decomp}
\widetilde{\X}\simeq \bigsqcup_{\Theta\subset\Delta}
2^{\#\Theta}(G/P_\Theta(K)) \quad \text{(disjoint
union)},
\eq
the union being over subsets of $\Delta$, where $\#\Theta$ is the number of elements of $\Theta$,
and $2^{\#\Theta}(G/P_\Theta(K))$ denotes the disjoint union
of $2^{\#\Theta}$ copies of the homogeneous space $G/P_\Theta(K)$, where  $P_\Theta(K)$ is a certain closed subgroup of $G$ associated to  $\Theta$. In particular,  for $\Theta=\Delta$, one has $P_\Delta(K)=K$, while for $\Theta=\emptyset$, $P_\emptyset(K)=P$.  In what follows, denote by $\widetilde \X_\Theta$ a component in $\widetilde \X$ isomorphic to $G/P_\Theta(K)$. 
The orbital decomposition is of normal crossing type, meaning that  for every point in $\widetilde{\X}$ there exists a
local coordinate system $(n_1,\dots , n_k,t_1,\dots ,
t_l)$ in a neighbourhood of that point such that two
points with coordinates $(n_1,\dots ,n_k,t_1,\dots ,t_l)$ and
$(n'_1,\dots ,n'_k,t'_1,\dots ,t'_l)$ belong to the same
$G$-orbit if, and only if, $\sgn t_j=\sgn t'_j$  for all
$j=1,\dots, l$. For a detailed description of $\widetilde \X$, the reader is referred to Part I.
On $\widetilde \X$, there is a natural representation 
of $G$ given by
\bqn 
\pi(g) \phi(\tilde x) =\phi(g^{-1} \cdot \tilde x),
\qquad \phi \in \mathrm{C}(\widetilde \X),
\eqn
where $\mathrm{C}(\widetilde \X)$ denotes the Banach space of continuous, complex-valued functions on $\widetilde \X$.  Let  $d_G$ be a Haar measure on $G$, and  $\S(G)$  the space of rapidly decreasing functions on $G$ introduced in Part I, Definition 1. Let further $\Omega $ be the density bundle on $\widetilde \X$, and consider for every $f\in \S(G)$  the continuous linear operator
\begin{equation*}
\pi(f):\Cinft(\widetilde \X) \longrightarrow
\Cinft(\widetilde \X) \subset \D'(\widetilde \X),
\end{equation*}
with Schwartz kernel given by the distributional section
$\mathcal{K}_f \in \D'(\widetilde \X \times \widetilde
\X, {{\bf 1}} \boxtimes \Omega)$.
Let  $$\mklm{(\widetilde{U}_{m_w},\phi^{-1}_{m_w})}_{w \in W}$$ be the finite atlas of $\widetilde \X$ constructed in Part I, where $W=M^\ast/M$ denotes the Weyl group of $(\g,\a)$, and $m_w \in M^\ast$, a representative of $w \in W$. The coordinates on each of the charts of this atlas are then precisely of the form $(n,t)=(n_1, \dots, n_k, t_1,\dots, t_l)$ described above.  For each point $\tilde x \in \widetilde \X$, choose open neighborhoods  $\widetilde W_{\tilde x}\subset \widetilde W_{\tilde x}'$ of $\tilde x$ contained in a chart $\widetilde{U}_{m_w(\tilde x)}$.   Since $\widetilde \X$ is compact, we can find a finite subcover of the cover $\mklm{\widetilde W_{\tilde x}}_{\tilde x\in \widetilde \X}$,  and in this way obtain a finite atlas $\mklm{\widetilde W_\gamma,\phi^{-1}_\gamma}_{\gamma\in I}$ of $\widetilde \X$, where  for simplicity we wrote $\phi_\gamma=\phi_{m_w(\tilde x)}$. Further, let $\mklm{\alpha_\gamma}_{\gamma\in I}$ be a partition of unity subordinate to the above atlas, and let $\mklm{\bar \alpha_\gamma}_{\gamma\in I}$ be another set of functions satisfying $ \bar \alpha_\gamma \in \CT(\widetilde W_\gamma')$ and $\bar \alpha_{\gamma|\widetilde W_\gamma}\equiv 1$.  Consider now the localization of $\pi(f)$ with respect to the atlas above given by
 \bqn 
A_{f}^{\gamma} u=[\pi(f)_{|\widetilde W_{\gamma}}(u\circ
\phi_{\gamma}^{-1})]\circ \phi_{\gamma}, \qquad
u\in\CT(W_{{\gamma}}), \,
W_{\gamma}=\phi^{-1}_{\gamma}(\widetilde W_\gamma)\subset \R^{k+l}.
 \eqn 
  Writing $\phi_{\gamma}^{g}= \phi_{\gamma}^{-1}\circ g^{-1}\circ \phi_{\gamma}$ and $x=(x_{1},\dots,x_{k+l})=(n,t)\in {W}_\gamma$ we obtain
\begin{equation*}
A_{f}^{\gamma} u(x)=\int_{G}  f(g)[(u\circ \phi_{\gamma}^{-1}) \bar \alpha_\gamma](g^{-1} \cdot \phi_\gamma(x)) \, d_G(g)=\int_{G}  f(g) c_{\gamma}(x,g)(u\circ \phi_{\gamma}^{g})(x)d_G(g),
\end{equation*}
where we put $c_{\gamma}(x,g)=\bar \alpha_\gamma(g^{-1} \cdot \phi_\gamma(x))$. Next, define the smooth functions
\begin{equation}
 a_{f}^{\gamma}(x,\xi)=\int_{G}e^{i(\phi_{\gamma}^{g}(x)-x )\cdot\xi} c_{\gamma}(x,g)f(g) d_G(g),
\end{equation}
and let $T_x$ be the diagonal $(l\times l)$-matrix with entries $x_{k+1}, \dots, x_{k+l}$. Introduce the auxiliary symbol 
\bq
\label{eq:auxsym}
\tilde{a}_f^{\gamma}(x,\xi)=a_{f}^{\gamma}(x,({\1}_k\otimes T^{-1}_{x})\xi)=\int_{G}e^{i\Psi_\gamma(g,x)\cdot  \xi} c_{\gamma}(x,g)f(g) d_G(g),
 \eq 
 where we put
\begin{equation*}
\Psi_\gamma(g,x)=[({\bf 1}_k\otimes T^{-1}_{x}) (\phi_{\gamma}^{g}(x)-x )]=(x_1(g\cdot \tilde x)-x_1(\tilde x),\dots,x_k(g\cdot \tilde x)-x_k(\tilde x),\chi_{{}_1}(g,\tilde x)-1,\dots,\chi_{{}_l}(g,\tilde x)-1),
\end{equation*}
the $\chi_j(g,\tilde x)$ being analytic functions, see Part I, equation (28). 
 One of the main results of Part I is  the following 
 \begin{theorem}
\label{thm:I}
Let $f \in \S(G)$. The restrictions of the operators $\pi(f)$ to the manifolds with corners
$\overline{\widetilde \X_{\Delta}}$ are totally
characteristic pseudodifferential operators of class
$\L^{-\infty}_b$. More precisely, the  operators
$\pi(f)$ are locally of the form \footnote{Here and in what follows we shall adhere to  the convention that, if not specified otherwise, integration is to be performed over  whole  Euclidean space $\rn$, with $n$ appropriate. In addition, we shall use the notation $\dbar \xi=(2\pi)^{-n} d\xi$. }
\begin{gather}
\label{20}
A^\gamma_fu(x)= \int e ^{i x \cdot \xi}
a_f^\gamma(x,\xi)\hat u(\xi) \dbar\xi, \qquad u \in
\CT(W_\gamma),
\end{gather}
where $a_f^\gamma(x,\xi)=\tilde a_f^\gamma(x,\xi_1,
\dots, \xi_k, x_{k+1} \xi_{k+1}, \dots, \xi_{k+l}
x_{k+l})$, and $\tilde a_f^\gamma(x,\xi) \in
\Symsl(W_\gamma \times \R^{k+l}_\xi)$ is a lacunary symbol given by \eqref{eq:auxsym}.  In particular, the kernel of the
operator $A^\gamma_f$ is determined by its restrictions
to $W_\gamma^\ast \times W_\gamma^\ast$, where
$W_\gamma^\ast=\{ x \in W_\gamma: x_{k+1} \cdots x_{k+l}
\not=0\}$, and given by the oscillatory integral
\begin{equation}
\label{20b}
K_{A_f^\gamma} (x,y)=\int e^{i(x-y) \cdot \xi}
a^\gamma_f(x,\xi) \dbar \xi.
\end{equation}
\end{theorem}

\begin{proof}
See Part I, Theorem 2.
\end{proof}

\section{Regularized traces}

\label{Sec:3}
We shall now define a regularized trace for the convolution operators $\pi(f)$ introduced in the previous section.  To begin with, note that, as a
consequence of Theorem \ref{thm:I}, we can write the
kernel of $\pi(f)$ locally in the form
\begin{align}
\label{27}
\begin{split}
  K_{A_f^\gamma}(x,y) &= \int e ^{i(x-y) \cdot  \xi} a^\gamma _f (x,\xi) \dbar \xi=\int e^{i(x-y) \cdot  ({\bf{1}}_k \otimes T_x^{-1}) \xi} \tilde a_f^\gamma(x,\xi) |\det ({\bf{1}}_k \otimes T_x^{-1})'(\xi)| \dbar \xi\\
&=\frac 1 {|x_{k+1}\cdots x_{k+l}|} \tilde A_f^\gamma(x,x_1-y_1, \dots, 1 - \frac{y_{k+1}}{x_{k+1}}, \dots), \qquad x_{k+1}\cdots x_{k+l} \not=0, 
\end{split}
\end{align}
where   $\tilde A_f^\gamma(x,y)$ denotes the inverse Fourier transform of the lacunary symbol $\tilde a_f^\gamma(x,\xi)$ given by
\begin{equation}
\label{eq:34}
  \tilde A_f^\gamma(x,y)= \int e ^{i y\cdot \xi} \tilde a_f ^\gamma(x,\xi) \, \dbar \xi.
\end{equation}
Since for $x \in W_\gamma$, the amplitude  $\tilde a_f^\gamma(x,\xi)$ is rapidly falling in $\xi$, it follows that $\tilde A_f^\gamma(x,y) \in \S(\R^n_y)$,  the Fourier transform being an isomorphism on the Schwartz space. Therefore, $K_{A^\gamma_f}(x,y)$ is rapidly decreasing as $| x_{j}| \to 0 $ for $x_j\not=y_j$ and   $k+1\leq j\leq  k+l$.  Furthermore, by the lacunarity of $\tilde a ^\gamma_f$, $ K_{A_f^\gamma}(x,y)$ is also rapidly decaying  as $| y_{j}| \to 0 $, $x_j\not=y_j$ and   $k+1\leq j\leq  k+l$.

Consider now the partition of unity $\{\alpha_{\gamma}\}$ subordinate to the atlas $\{(\widetilde W_{\gamma},
\phi^{-1}_{\gamma})\}$.  By equation \eqref{27}, the
restriction of the kernel of $A_f^{\gamma}$ to the
diagonal is given by
$$K_{A_f^\gamma}(x,x)=\frac 1 {|x_{k+1}\cdots x_{k+l}|}
\tilde A_f^\gamma(x,0), \qquad x_{k+1}\cdots x_{k+l}
\not=0.$$
These restrictions yield a family of smooth functions
$k_f^\gamma(\tilde x)=K_{A_f^\gamma}(\phi^{-1}_{\gamma}(\tilde x),\phi^{-1}_{\gamma}(\tilde x))$
which define a density $k_f$
 on 
\bqn
 2^{\#l}(G/K)\subset \widetilde \X.
\eqn
Nevertheless, the functions $k_f^\gamma(\tilde x)$ are not
locally integrable on the entire compactification $\widetilde \X$, so that
we cannot define a trace of $\pi(f)$ by integrating the
density $k_f$ over the diagonal $\Delta_{\widetilde \X
\times \widetilde \X} \simeq \widetilde \X$. Instead, we have the following

\begin{proposition}
\label{prop:1}
Let $f \in \S(G)$, $s \in \C$, and define for $\Re s>0$ 
\begin{align*}
\Tr_s \pi(f)&= \sum _\gamma \int _{W_\gamma}
(\alpha_\gamma \circ \phi_{\gamma})(x) |x_{k+1} \cdots
x_{k+l}|^s \widetilde A_f^\gamma(x,0) dx\\
&= \eklm{|x_{k+1} \cdots x_{k+l}|^s,\sum_\gamma
(\alpha_\gamma \circ \phi_{\gamma})\widetilde
A_f^\gamma(\cdot,0)}.
\end{align*}
Then $\Tr_s \pi(f) $ can be continued analytically to a 
meromorphic function in $s$ with at most poles at  $-1, -3, \dots$.
Furthermore, for $s \in \C-\mklm{-1,-3,\dots}$,
\begin{align}
\label{eq:Theta}
\Theta_\pi^s:\CT(G) \ni f \mapsto \Tr_s \pi(f) \in \C
\end{align}
defines a distribution density on $G$. 
\end{proposition}
\begin{proof}
The fact that $\Tr_s \pi(f)$ can be continued meromorphically  is a consequence of the analytic continuation of
$|x_{k+1} \cdots x_{k+l}|^s$ as a distribution in
$\R^{k+l}$, proved by  Bernshtein-Gel'fand in 
\cite{bernshtein-gelfand69}, Lemma 2. One even has that
\bqn 
\eklm{|x_{k+1}|^{s_1} \cdots |x_{k+l}|^{s_l},u}, \qquad u \in \CT(\R^{k+l}),
\eqn
can be continued meromorphically in the variables $s_1,
\dots, s_l$ to $\C^l$ with poles $s_i=-1,-3, \dots$. To see that \eqref{eq:Theta} is a distribution density, note that $\Theta_\pi^s: \CT(G) \rightarrow \C$ is certainly linear. Since   $|x_{k+1} \cdots x_{k+l}|^s$ is a distribution, for any  open, relatively compact subset  $\omega \subset \R^{k+l}$ there exist $C_\omega>0$ and $B_\omega\in \N$ such that 
\bq
\label{eq:11}
|\eklm{|x_{k+1} \cdots x_{k+l}|^s,u}| \leq C_\omega \sum_{|\beta|\leq B_\omega} \sup|\gd^\beta u|, \qquad u \in \CT(\omega).
\eq
 Let  now $\mathcal{O}\subset G$ be an arbitrary open, relatively compact subset, and  $f \in \CT(\mathcal{O})$.  With equation \eqref{eq:34} one has
\bq
\label{eq:11a}
\Tr_s \pi(f)= \eklm{|x_{k+1} \cdots x_{k+l}|^s,\sum_\gamma
(\alpha_\gamma \circ \phi_{\gamma})\int \tilde a_f^\gamma(\cdot ,\xi)  \dbar \xi}.
\eq
By equation (38) of Part I, one computes for arbitrary $N \in \N$ that
\bqn 
e^{i\Psi_\gamma(g,x) \cdot \xi}=\frac1{(1+|\xi|^2)^N} \sum_{r=0}^{2N} \sum_{|\alpha|=r} b_\alpha^N(x,g) dL(X^\alpha) \Big [ e^{i\Psi_\gamma(g,x) \cdot \xi} \Big ],
\eqn
where the coefficients $b_\alpha^N(x,g)$ are smooth, and at most of exponential growth in $g$. With \eqref{eq:auxsym} and Proposition 1 of Part I we  therefore obtain for $\tilde a_f^\gamma(x,\xi)$ the expression
\bqn 
\tilde a_f^\gamma(x,\xi)= \frac1{(1+|\xi|^2)^N}\int_G e^{i\Psi_\gamma(g,x) \cdot \xi} \sum_{r=0}^{2N} \sum_{|\alpha|=r} (-1)^r dL(X^{\tilde \alpha}) \Big [ b_\alpha^N(x,g)c_\gamma(x,g) f(g)  \Big ] d_G(g).
\eqn
Inserting this in \eqref{eq:11a}, and taking  $N$ sufficiently large, we obtain with \eqref{eq:11} that 
\begin{align*}
|\Tr_s \pi(f) | \leq 
 C_{\mathcal{O}} \sum_{|\beta| \leq B_{ \mathcal{O}}}  \sup| dL(X^\beta) f|
\end{align*}
for suitable $C_{\mathcal{O}}>0$ and $B_{\mathcal{O}}\in \N$. Since the universal enveloping algebra $\U(\g_\C)$ can be identified with the algebra of invariant differential operators on $G$, the assertion now follows with  \cite{warner72}, page 480.  
\end{proof}

\begin{remark}
Using Hironaka's theorem on resolution of singularities,
Bernshtein-Gel'fand \cite{bernshtein-gelfand69} and
Atiyah \cite{atiyah70} even proved the following 
general result. 
Let $M$ be a real analytic manifold and $f$ a non-zero,
real analytic function on $M$. Then $|f|^s$, which is
locally integrable for $\Re s > 0$, extends analytically
to a distribution on $M$ which is a meromorphic
function of $s$ in the whole complex plane. The poles
are located at the negative rational numbers, and their
order does not exceed the dimension of $M$. From this
one deduces
that if $f : M \rightarrow \C$ is a non-zero analytic
function, then there exists a distribution $S$ on $M$
such that $fS=1$. This is the
H\"{o}rmander-Lojasiewicz theorem on the division of
distributions, and implies the existence of temperate
fundamental solutions for constant-coefficient
differential operators.
\end{remark}

Consider next
the Laurent expansion of $\Theta_\pi^s(f)$ at $s=-1$.
For this, let $u \in \CT(\R^{k+l})$ be a test function,
and consider the expansion
\bqn 
\eklm{|x_{k+1} \cdots x_{k+l}|^s,u }= \sum_{j=-q}^\infty
S_j(u) (s+1)^j,
\eqn
where $S_k \in \D'(\R^{k+l})$. Since $|x_{k+1} \cdots
x_{k+l}|^{s+1}$ has no pole at $s=-1$, we necessarily
must have
\bqn 
|x_{k+1} \cdots x_{k+l}| \cdot S_j =0 \quad \text{for }\,
j<0, \qquad |x_{k+1} \cdots x_{k+l}| \cdot S_0=1
\eqn
as distributions. Therefore $S_0 \in \D'(\R^{k+l})$
represents a distributional inverse of $|x_{k+1} \cdots x_{k+l}|$.
By repeating the reasoning of the proof of Proposition \ref{prop:1} we  arrive at the following
\begin{proposition}
For $f \in \S(G)$,  let  the regularized trace of the operator $\pi(f)$ be defined by
\begin{align*}
\Tr_{reg}\pi(f)&=\left\langle S_0,\sum_{\gamma}
(\alpha_{\gamma}\circ \phi_{\gamma} ) \tilde
A_f^\gamma(\cdot,0) \right\rangle.
\end{align*}
Then $\Theta_\pi: \CT(G) \ni f \mapsto \Tr_{reg} \pi(f) \in \C$  constitutes  a distribution density on  $G$, which is called  the \emph{character of the representation $\pi$}. 
\end{proposition}
\qed

\begin{remark}
An alternative definition of $\Tr_{reg} \pi(f)$ could be given within the calculus of b-pseudodifferential operators developed by Melrose. For a detailed description, the reader is referred to  \cite{loya98}, Section 6.
\end{remark}

{In what follows, we shall identify  distributions with distribution densities on $G$ via the Haar measure $d_G$.} Our next aim is to understand the distributions $\Theta^s_\pi$ and $\Theta_\pi$ in terms of the $G$-action on $\widetilde \X$. We shall actually show that on a certain open set of transversal elements, they are represented by locally integrable functions  given in terms of fixed points. Similar expressions where derived by Atiyah and Bott  for the global character of an induced representation of $G$. Their work is based on the concept of  transversal trace of a pseudodifferential operator, and  will be explained in the next section.

\section{Transversal trace and characters of induced representations}
\label{sec:AB}

In \cite{atiyah-bott67}, Atiyah and Bott extended the classical Lefschetz fixed point theorem to geometric endomorphisms on elliptic complexes. Their work relies on the concept of transversal trace of a smooth operator, and its extension by continuity to pseudodifferential operators. The Lefschetz theorem then follows by showing that the  Lefschetz number of a geometric endomorphism is given by an alternating sum of transversal traces, and extending an analogous alternating sum formula for smooth endomorphisms. To explain the notion of 
 transversal trace of a pseudodifferential operator, let us introduce the following 

\begin{definition}
Let $M$ be a smooth manifold. A fixed point $x_0$ of a
smooth map $f: M \rightarrow M$ is said to be
\emph{simple} if $\det(\1-df_{x_0}) \neq 0$, where
$df_{x_0}$ denotes the differential of $f$ at $x_0$. The map $f$ is called \emph{transversal} if
it has only simple fixed points. 
\end{definition}
Note that the non-vanishing
condition on the determinant is equivalent to the
requirement that the graph of $f$ intersects the diagonal
transversally at $(x_0,x_0) \in M \times M$, 
and hence the terminology.  In particular,  a
simple fixed point is an isolated fixed point.
 Let now $U$ be an open subset of
$\R^n$, $V$ open in $U$, and consider a smooth map
$\alpha:V\rightarrow U$ with a simple fixed point at
$x_0$. We choose $V$  so small, that $x\mapsto
x-\alpha(x)$ defines a diffeomorphism of $V$ onto its
image. Let $\Lambda:V \rightarrow U \times U$ be the map
$\Lambda(x)=(\alpha(x),x)$, and assume that $A \in
\L^{-\infty}(U)$ is a smooth operator with symbol $a(x,\xi)$. The kernel $K_A$ of $A$ is a smooth function
on $U\times U$, and its restriction
$\Lambda^\ast K_A$ to the graph of $\alpha$ defines a
distribution on $V$ according to
\begin{align}
\label{J}
\begin{split}
\eklm{\Lambda^\ast K_A, v} &=\int \int
e^{i(\alpha(x)-x)\cdot \xi }a(\alpha(x),\xi) v(x) \, \dbar \xi
\, dx\\ &=\int \int e^{-iy\cdot \xi} \frac {
a(\alpha(x(y)), \xi) v (x(y))}{|\det (\1 -d\alpha
(x(y)))|} \, dy \, \dbar \xi, \qquad v \in \CT(V),
\end{split}
\end{align}
where we made the substitution $y=x-\alpha(x)$, and the
change in order of integration is permissible
because $ a(x,\xi) \in \Syms(U)$. Now, for  $a(x,\xi) \in \Sym^l(U)$, we observe that by differentiating 
$$\int e^{-iy\cdot \xi} a(\alpha(x(y)), \xi)\frac {v (x(y))}{|\det (\1 -d\alpha
(x(y)))|} dy $$ with respect to $\xi$, and integrating by parts with respect to $y$, we obtain the estimate 
\bqn
|\partial_\xi^\gamma \int e^{-iy\cdot \xi} a(\alpha(x(y)), \xi)\frac {v (x(y))}{\left|\det (\1 -d\alpha(x(y)))\right|} dy | \leq C \eklm{\xi}^{l-|\beta|}
\eqn
for arbitrary multi-indices $\gamma$ and $\beta$ and some constant $C>0$. Thus,  as an oscillatory integral, the last expression in \eqref{J} defines a
distribution on $V$ for any $a(x,\xi) \in \Sym^l(U)$. The distribution $\Lambda^\ast K_A$ is called \emph{the
transversal trace of $A\in \L^l(U)$}. If, in particular,
$a(x,\xi)=a(x)$ is a polynomial of degree zero in $\xi$,
one computes that 
\begin{equation}
\label{eq:12}
\Lambda^\ast K_A=\frac{ a(x_0) \delta_{x_0}}{|\det ( 1
-d\alpha (x_0))|}.
\end{equation}
This diccussion can be globalized. Let ${\bf X}$ be a smooth manifold, $E$ a vector bundle over ${\bf X}$, $\alpha:{\bf X} \rightarrow {\bf X}$ a $\Cinft$-map with only simple fixed points, and
\begin{displaymath}
  A:\Gamma_c(\alpha^\ast E) \longrightarrow \Gamma(E)
\end{displaymath}
a pseudodifferential operator of order $l$ between smooth sections. Denote the density bundle on ${\bf X}$ by $\Omega$, put  $F=\alpha^\ast E$, and define $F'=F^\ast \otimes \Omega$. The kernel $K_A$ is then a distributional section of $E\boxtimes F'$. In other words, $K_A \in \D'(E \boxtimes F')=\D'({\bf X}\times {\bf X}, E \boxtimes F')$. 
 Similarly, one has $K_{\alpha^\ast A} \in \D'({\bf X}\times {\bf X}, F\boxtimes F')$, where $\alpha^\ast A$ denotes the composition
\begin{displaymath}
  \alpha^\ast A:\Gamma_c(F)\stackrel{A}{\longrightarrow} \Gamma(E)\stackrel{\alpha^\ast}{\longrightarrow} \Gamma(F).
\end{displaymath}
 If $A \in \L^{-\infty}(F,E)$,  $K_A$ is a smooth section on ${\bf X}\times {\bf X}$, and $K_A(x,y) \in E_x\otimes F_y'$. In this case, $
  K_{\alpha^\ast A} (x,y) =K_A(\alpha(x),y),
$
so that one deduces $K_{\alpha^\ast A}(x,x)\in E_{\alpha(x)}\otimes F'_x=F_x \otimes (F^\ast \otimes \Omega)_x \simeq \Lcal(F_x,F_x)\otimes \Omega_x$. As a consequence, $\Tr  K_{\alpha^\ast A}(x,x)$ becomes a section of $\Omega$, where $\Tr$ denotes the bundle homomorphism
\begin{equation}
\label{K}
  \Tr: F\otimes F' \longrightarrow \Omega. 
\end{equation}
Hence, if $\bf X$ is compact,   one can define the trace of $\alpha^\ast A$ as 
\begin{equation*}
  \Tr \alpha^\ast A=\int _{\bf X}  \Tr K_{\alpha^\ast A}(x,x).
\end{equation*}
This trace can be extended to arbitrary $A \in \L^l({\bf X})$. Indeed, let $\Delta$ be the diagonal in ${\bf X}\times {\bf X}$, and denote the canonical isomorphism $\Delta \simeq {\bf X}$ also by $\Delta$.  The foregoing local considerations  imply that the map $\Theta: \Lcal(\E'(F),\Gamma(E)) \rightarrow \Gamma(F\otimes F')$ given by $A \mapsto \Delta^\ast K_{\alpha^\ast A}=K_{\alpha^\ast A}(x,x)$ has an extension
  \begin{equation*}
    \Theta:\L^l(F,E) \longrightarrow \D'(F\otimes F')
  \end{equation*}
which is continuous with respect to the strong operator topology  on bounded sets of $\L^l(F,E)$, see \cite{atiyah-bott67}, Proposition 5.3. Since the bundle homomorphism \eqref{K} induces continuous linear maps
\begin{equation*}
  \Tr:\Gamma (F\otimes F') \longrightarrow \Gamma(\Omega), \qquad \Tr: \D'(F\otimes F') \longrightarrow \D'(\Omega), 
\end{equation*}
where $\D'(\Omega)=\D'({\bf X},\Omega)=\Gamma_c(\Omega^\ast \otimes \Omega)'=\Gamma_c(1)'=\CT({\bf X})'$  is the space of distribution densities on $\bf X$, we see  that $\Tr \Theta(A)$ can be defined for any $A \in L^l(F,E)$ in a unique way.  Consequently, for compact $\bf X$, the map $\L^{-\infty}(F,E) \rightarrow \C, A \to \Tr \alpha^\ast A$ has a unique continuous extension
  \begin{equation*}
    \Tr_\alpha:\L^l(F,E) \longrightarrow \C, \qquad A \mapsto \Tr_\alpha A=\eklm{\Tr \Theta(A),1},
  \end{equation*}
  called the \emph{transversal trace of $A$}. In the case that $A$ is induced by a bundle homomorphism $\phi$, it follows from \eqref{eq:12} that 
  \bq
  \label{eq:M}
  \Tr_\alpha A=\sum_{x \in \Fix(\alpha)}\nu_x(A), \qquad 
  \nu_x(A)= \frac{\Tr \phi_x}{|\det(\1-d\alpha(x))|},
  \eq
  the sum being over the fixed points of $\alpha$ on $\bf X$,
  see \cite{atiyah-bott67}, Corollary 5.4. 
  
 In the context of representation theory, this trace was employed by Atiyah and Bott in \cite{atiyah-bott68} to compute the global character of an induced representation.  
  Thus, let $G$ be a Lie group,  $H$ a closed subgroup of $G$, and  $\rho$ a representation of $H$ on a finite dimensional vector space $V$. The representation of $G$ induced by $\rho$ is  a geometric endomorphism in the space of sections over $G/H$ with values in the homogeneous vector bundle $G\times_HV$, and shall be denoted by $T(g)=(\iota_\ast \rho)(g)$. Assume that  $G/H$ is compact,  and let $d_G$ be a Haar measure on $G$. Consider a compactly supported smooth function  $f\in \CT(G)$, and the corresponding convolution operator $T(f)=\int _G f(g) T(g) d_G(g)$. It  is a smooth operator, and, since $G/H$ is compact,  has a well defined trace. Consequently, the map  
 \bqn 
 \Theta_T:\CT(G) \ni f \longmapsto \Tr T(f)\in \C
 \eqn
  defines a distribution on $G$ called the \emph{distribution character   of the induced representation $T$}. On the other hand, assume that   $g\in G$ is such that $l_{g^{-1}}:G/H\rightarrow G/H, xH \mapsto g^{-1} xH$, has only simple fixed points. In this case, a transversal trace     $\Tr^\flat T(g)$ of  $T(g) $ can be defined according to 
  \bqn 
  \Tr^\flat T(g)=\Tr_{l_{g^{-1}}}(\Gamma(\phi_g)),
  \eqn  
  where $\phi_g: l_{g^{-1}}^\ast (G\times _HV) \rightarrow G\times_HV$ is the endomorphism associated to $T(g)$ such that 
 \bqn 
 T(g)= \phi_g \circ l_{g^{-1}}^\ast,
 \eqn
 and $\Gamma(\phi_g): \Gamma(l_{g^{-1}}^\ast (G\times _HV)) \rightarrow \Gamma(G\times_HV)$. $  \Tr^\flat T(g)$ is given by a sum over fixed points of $g$, and one can show that, on an open set $G_T\subset G$,
\begin{equation}
\label{IV}
 \Theta_T(f)=\int_{G_T} f(g) \Tr^\flat T(g) d_G(g), \qquad f \in \CT(G_T).
\end{equation}
Thus, the distribution character of a parabolically induced representation of a Lie group $G$  is represented on $G_T$ by the transversal trace of the corresponding geometric endomorphism. If $G$ is compact, the Lefschetz theorem reduces to the Hermann--Weyl formula by the theory of Borel and Weil. It can be interpreted as expressing the character of a finite dimensional representation as an alternating sum of characters of infinite dimensional representations. 
In what follows, we shall prove similar formulae for the distributions $\Theta_\pi$ and $\Theta_\pi^s$ defined in the previous section, after reviewing some largely known facts about group actions on homogeneous spaces.

\section{Fixed point actions on homogeneous spaces }
\label{Sec:4}

Let $G$ be a Lie group with Lie algebra $\g$,  $H
\subset G$  a closed subgroup with Lie algebra $\h$, and $\pi:G\rightarrow G/H$ the canonical projection. For an element $g \in
G$, consider the natural left action  $l_{g}:G/H \rightarrow G/H$  given by $l_{g}(xH)=gxH$. Let $\Ad^G$ denote the adjoint action of $G$ on
$\g$. We begin with two well-known lemmata, see e.g. \cite{atiyah-bott68}, page 463.
\begin{lemma}
\label{lemma 1}
$l_{g^{-1}}:G/H \rightarrow G/H$ has a fixed point if
and only if $g \in \bigcup_{x \in G}xHx^{-1}$. Moreover, to every fixed point $xH$ one can associate
a unique conjugacy class $h(g,xH)$ in $H$.
\end{lemma}
\begin{proof}
Clearly, 
$$l_{g^{-1}}(xH)=xH \Longleftrightarrow g^{-1}xH=xH
\Longleftrightarrow (g^{-1}x)^{-1}x \in H \Longleftrightarrow
x^{-1}gx=h(g,x),$$
 where $h(g,x) \in H$. So $l_{g^{-1}}$
has a fixed point $xH$ if, and only if, $g \in \bigcup_{x \in G}xHx^{-1}$. Now, if $y \in G$ is 
such that $xH=yH$, then $y=xh$ for some $h \in H$. This
gives us that
$h(g,y)=y^{-1}gy=(xh)^{-1}g(xh)=h^{-1}(x^{-1}gx)h=h^{-1}h(g,x)h$.
Thus, as $x$ varies over representatives of the coset $xH$, $h(g,x)$ varies over a conjugacy
class $h(g,xH)$ in $H$.
\end{proof}
 \begin{lemma}
\label{prop:transversality}
Let $xH$ be a fixed point of $l_{g^{-1}}$ and let $h \in
h(g,xH)$. Then
$$
\det (\1 - dl_{g^{-1}})_{xH}=\det(\1-\Ad^G_H(h)),
$$
where $\Ad^G_H : H\rightarrow \mathrm{Aut}(\g / \h) $ is the
isotropy action of H on $\g / \h$.
\end{lemma}
\begin{proof}
Let $L_g$ and $R_g$ be the left and right translations, respectively,  of $g\in G$ on $G$. 
We begin with the  observation that 
\bq
\label{eq:2}
\pi\circ L_{g^{-1}}=l_{g^{-1}}\circ \pi,
\eq
where $\pi$ is the natural map from $G$ to $G/H$. Let $e$ be the identity in $G$, and 
$T_{\pi(e)}(G/H)$  the tangent space to $G/H$ at
the point $\pi(e)$. The
derivative $d\pi : \g \rightarrow T_{\pi(e)}(G/H)$ is a
surjective linear map with kernel $\h$, and therefore induces an
isomorphism between $\g / \h$ and $T_{\pi(e)}(G/H)$,
which we shall again denote by $d\pi$. Notice also that, for $h \in H$, $\Ad^G(h)$ leaves $\h$
invariant and so induces a map $\Ad^G_H(h):\g / \h
\rightarrow \g / \h$.
Now, let $xH$ be a fixed point of $l_{g^{-1}}$,  and take 
$h \in h(g,xH)$. Choose $x$ in the coset $xH$ such that
$g^{-1}x=xh$. For  $y \in G$ one computes  
$$
(\pi \circ
L_{g^{-1}} \circ R_{h^{-1}})(y)=\pi
(g^{-1}yh^{-1})=g^{-1}yH=l_{g^{-1}}(yH)=(l_{g^{-1}}
\circ \pi) (y),
$$ 
so that
\bq
\label{eq:3}
\pi \circ L_{g^{-1}} \circ R_{h^{-1}}=l_{g^{-1}} \circ
\pi.
\eq
Observe, additionally, that $L_{g^{-1}} \circ R_{h^{-1}}$ fixes
$x$. We therefore see that
$L_{g^{-1}} \circ R_{h^{-1}} \circ L_x=L_x \circ L_h
\circ R_{h^{-1}}$, which, together with equations
\eqref{eq:2} and \eqref{eq:3}, leads us to 
\bq
l_x \circ \pi \circ L_h \circ R_{h^{-1}} = l_{g^{-1}}
\circ l_x \circ \pi.
\eq
Differentiating this, and using the identification $dl_x
\circ d\pi : \g / \h\rightarrow T_{\pi(x)}(G/H)$, we
obtain the commutative diagram
$$\begin{CD}
  \g / \h @>{\Ad^G_H(h)}>> \g / \h \\
    @V{dl_x \circ d\pi}VV                     @V{dl_x\circ d \pi}VV \\
  T_{\pi(x)}(G/H)   @>{dl_{g^{-1}}}>> T_{\pi(x)}(G/H) 
\end{CD}$$
 thus proving the lemma.
\end{proof}

Consider now the case when $G$ is a connected, real, semi-simple Lie group with finite centre, $\theta$ a Cartan involution of $\g$, and $\g=\k\oplus\p$ the corresponding Cartan decomposition. Further, let  $K$ be
 the maximal compact subgroup of $G$ associated to $\k$,  and consider the corresponding Riemannian symmetric space  $\X=G/K$ which is assumed to be of non-compact type. By definition, $\theta$ is an involutive automorphism of $\g$ such
that the bilinear form $\langle \cdot , \cdot \rangle _{\theta}$ is
strictly positive definite. In particular, 
$\langle\cdot,\cdot\rangle _{\theta|\p\times\p}$ is a
symmetric, positive-definite, bilinear form, yielding a left-invariant  metric on $G/K$. 
Endowed with this metric, $G/K$ becomes a complete, simply connected,
Riemannian manifold with non-positive sectional
curvature. Such manifolds are called \emph{Hadamard manifolds}. 
 Furthermore, for each $g \in G$, $l_{g^{-1}}:G/K \rightarrow G/K$ is an isometry on $G/K$ with respect to this left-invariant metric. Note that  Riemannian symmetric spaces of {non-compact type} are
precisely the simply connected Riemannian symmetric
spaces with sectional curvature $\kappa \leq 0$ and with
no Euclidean de Rham factor. We then  have the following

\begin{lemma}
Let  $g \in G$ be such that $l_{g^{-1}}:G/K \rightarrow G/K$ is transversal.
Then  $l_{g^{-1}}$ has a
unique fixed point in $G/K$.
\end{lemma}
\begin{proof}
Let $M$ be a Hadamard manifold, and $\phi$ an isometry on $M$ that leaves two distinct points  $x,y \in M$ fixed. By general theory, there is a unique minimal geodesic $\gamma:\R \rightarrow M$ joining $x$ and $y$. Let  $\gamma(0)=x$ and  $\gamma(1)=y$, so that  $\phi \circ \gamma(0)=\phi (x)=x$ and
$\phi \circ \gamma(1)=\phi(y)=y$. Since isometries
take geodesics to geodesics, $\phi \circ \gamma$ is
a geodesic in $M$, joining  $x$ and $y$.
By the uniqueness of $\gamma$ we  therefore conclude that  $\phi \circ
\gamma=\gamma$. This means that  an isometry on a Hadamard
manifold with two distinct fixed points also fixes the
unique geodesic joining them point by point. Since, by assumption, $l_{g^{-1}}:G/K \rightarrow G/K$ has  only isolated fixed points, the lemma follows. 
\end{proof}
 In what follows, we shall call an element $g \in G$ \emph{transversal relative to  a closed subgroup $H$} if $l_{g^{-1}}:G/H \rightarrow G/H$ is transversal, and  denote the set of all such elements by $G(H)$. 
\begin{proposition}
Let $G$ be a connected, real, semi-simple Lie group with finite centre, and  $K$  a maximal compact subgroup of $G$.  Suppose $\rank(G)=\rank(K)$. Then any regular element of $G$ is transversal relative to $K$. In other words,  $G' \subset G(K)$, where $G'$ denotes the set of regular elements in $G$.
\end{proposition}
\begin{proof}
 If a regular element $g$ is such that $l_{g^{-1}}:G/K \rightarrow G/K$ has no fixed points, it is of course transversal. Let, therefore, $g \in G'$ be such that $l_{g^{-1}}$ has a fixed point $x_0K$. By Lemma \ref{lemma 1}, $g$ must be conjugate to an element $k(g,x_0)$ in $K$. Consider now   a maximal family of mutually non-conjugate Cartan subgroups  $J_1, \dots, J_r$ in $G$, and put $J_i'=J_i \cap G'$ for $i \in \{1,\dots,r\}$.  A result of Harish Chandra then implies  that 
\bqn
G'=\bigcup_{i=1}^r \bigcup_{x \in G} x\, J'_i\, x^{-1},
\eqn 
see \cite{warner72}, Theorem 1.4.1.7. 
From this we  deduce that 
\bqn 
g=xk(g,x_0)x^{-1}=yjy^{-1} \quad \text{ for some $x,y \in G, j\in J'_i$ for some $i$}.
\eqn
Hence, $k(g,x_0)$ must be regular. Now, let $T$ be a maximal torus of $K$. It is a Cartan subgroup of $K$, and the assumption that $\rank(G)=\rank(K)$ implies that  that $T$ is also Cartan in $G$.  Let  $k(g,x_0K)$ be the conjugacy class in $K$ associated to $x_0K$, as in Lemma \ref{lemma 1}.
As $K$ is compact, the maximal torus $T$  intersects every conjugacy class in $K$. Varying  $x_0$ over the coset $x_0K$,   we can therefore assume that  $k(g,x_0) \in k(g,x_0K) \cap T$. Thus, we conclude  that $k(g,x_0 ) \in T \cap G'$. Note that, in particular, we can choose $J_i=T$ by the maximality of the $J_1,\dots,J_r$. Now,  for a regular element $h \in G$ belonging to a Cartan subgroup $H$ one necessarily has $\det(\1-\Ad^G_H(h))\neq 0$, compare the  proof of Proposition 1.4.2.3 in  \cite{warner72}. Therefore $\det(\1-\Ad^G_T(k(g,x_0)))\neq 0$, and consequently, $\det(\1-\Ad^G_K(k(g,x_0)))\neq 0$. The assertion of the proposition now follows from Lemma \ref{prop:transversality}.
\end{proof}

\begin{corollary}
\label{cor:1}
Let $G$ be a connected, real, semi-simple Lie group with finite centre,   $K$  a maximal compact subgroup of $G$, and suppose that $\rank\, (G)=\rank\, (K)$. Then the set of transversal elements $G(K)$  is open and dense in $G$.
\end{corollary}
\begin{proof}
Clearly,  $G(K)$ is open. Since the set of regular elements $G'$ is dense in $G$, the corollary follows from the previous proposition. 
\end{proof}

\begin{remark}
\label{rem:2}
To close this section, let us remark that with  $G$ as above, and $P$ a parabolic subgroup of $G$, it is a classical result that $G' \subset G(P)$, see \cite{clozel}, page 51.
\end{remark}

\section{Character formulae}
\label{Sec:6}
Let the notation be as before. We are now in a position to describe the distributions $\Theta_\pi^s$ and $\Theta_\pi$ introduced in Section \ref{Sec:3}. Thus, let $(\pi,C(\widetilde \X))$ be the regular representation of $G$ on the Oshima compactification $\widetilde \X$ of the Riemannian symmetric space $\X= G/K$ of non-compact type, and denote by $\Phi_g(\tilde x)= g \cdot \tilde x$ the $G$-action on $\widetilde \X$. Let further  $G(\widetilde \X)\subset G$ be the set of elements $g$ in $G$ acting transversally on $\widetilde \X$.  
\begin{remark}
The set  $G(\widetilde \X)$ is open.  Corollary \ref{cor:1} and Remark \ref{rem:2} imply that $G(\widetilde \X) $ is dense  if $\rank \,( G/K)=1$, and non-empty if $\rank \, ( G/K)=2$, and $\rank\, (G)=\rank\, (K)$.
\end{remark}

\begin{theorem}  Let  $f \in \CT(G)$ have  support in $G(\widetilde \X)$, and $s \in \C$, $\Re s>-1$. Then
\bq
\label{eq:FPF}
\Tr_s \pi(f) =\int_{G(\widetilde \X)}  f(g) \left (\sum_{\tilde x \in \mathrm{Fix}(g)}  \sum _\gamma       \frac{ \alpha_\gamma (\tilde x) |x_{k+1}(\kappa_\gamma^{-1}(\tilde x)) \cdots
x_{k+l}(\kappa_\gamma^{-1}(\tilde x))|^{s+1} }{ |\det (1 -d\Phi_{g^{-1}}(\tilde x))|} \right ) d_G(g),
\eq
where $\Fix(g)$ denotes the set of fixed points of $g$ on $\widetilde \X$.  In particular, $\Theta_\pi^s:\CT(G) \ni f \to \Tr_s \pi(f)\in \C$ is regular on $G(\widetilde \X)$. 
\end{theorem}

\begin{proof}
By Proposition \ref{prop:1}, 
\begin{align*}
\Tr_s \pi(f)&=\sum _\gamma \int _{W_\gamma}
(\alpha_\gamma \circ \phi_{\gamma})(x) |x_{k+1} \cdots
x_{k+l}|^{s}   \widetilde A_f^\gamma(x,0) dx
\end{align*}
is a meromorphic function in $s$ with possible poles at $-1,-3,\dots$. Assume that $\Re s>-1$. Since $\alpha_\gamma \in \CT(\widetilde W_\gamma)$, and 
$
 \widetilde A_f^\gamma(x,0)=\int \tilde a_f^\gamma(x,\xi) \dbar \xi,
$
where $ \tilde a_f^\gamma(x,\xi)\in \Sym^{-\infty}_{la}(W_\gamma \times \R^{k+l})$ is rapidly decaying in $\xi$ by Theorem \ref{thm:I}, we can interchange the order of integration to  obtain
\begin{align*}
\Tr_s \pi(f)&=\sum _\gamma \int  \int _{W_\gamma}
(\alpha_\gamma \circ \phi_{\gamma})(x) |x_{k+1} \cdots
x_{k+l}|^{s} \tilde a_f^{\gamma}(x,\xi)  dx \, \dbar \xi.
\end{align*}
Let $\chi \in \CT(\R^{k+l}, \R^+)$ be equal $1$ in a neighborhood of $0$, and $\epsilon >0$. Then, by Lebesgue's theorem on bounded convergence,
\bqn 
\Tr_s \pi(f)=\lim_{\epsilon \to 0} I_\epsilon,
\eqn
where we defined
\bqn
I_\epsilon= \sum _\gamma \int  \int _{W_\gamma}
(\alpha_\gamma \circ \phi_{\gamma})(x) |x_{k+1} \cdots
x_{k+l}|^{s} \tilde a_f^{\gamma}(x,\xi)  \chi(\epsilon \xi) \, dx  \, \dbar \xi.
\eqn
Taking into account \eqref{eq:auxsym}, and interchanging the order of integration once more,  one sees that 
\begin{align*}
I_\epsilon&= \int_{G}   f(g) \sum _\gamma   \int \int _{W_\gamma}e^{i \Psi_\gamma(g,x) \cdot  \xi} c_{\gamma}(x,g)(\alpha_\gamma \circ \phi_{\gamma})(x) |x_{k+1} \cdots x_{k+l}|^{s} \chi(\epsilon \xi) dx \,  \dbar \xi \,  d_G(g),
\end{align*}
everything in sight being absolutely convergent. Let us now set
\bqn
 I_\epsilon(g)=  f(g) \sum _\gamma   \int \int _{W_\gamma}e^{i \Psi_\gamma(g,x) \cdot  \xi} c_{\gamma}(x,g) (\alpha_\gamma \circ \phi_{\gamma})(x) |x_{k+1} \cdots x_{k+l}|^{s} \chi(\epsilon \xi) dx \,  \dbar \xi,
\eqn
so that 
$
I_\epsilon= \int_G I_\epsilon(g) \, d_G(g). 
$
We would like to pass to the limit under the integral, for which we
are going to show that $\lim_{\epsilon \to 0} I_\epsilon(g)$ is an integrable function on $G$. For this, let us  fix an arbitrary  $g \in G(\widetilde \X)$. By definition, $g$ acts only with simple fixed points on $\widetilde \X$. Since each of them is isolated, $g$ can have at most finitely many fixed points on $\widetilde \X$. Consider therefore a cut--off function $\beta_g \in \Cinft(\widetilde \X,\R^+)$ which is  equal  $1$ in a small neighborhood of each fixed point of $g$,  and whose  support decomposes into a union of connected components each of them containing only one fixed point of $g$. By choosing the support of $\beta_g$ sufficiently close to the fixed points we can, in addition,  assume that 
\bq
\label{eq:21}
\det(d\Phi_g(\tilde x) -\1)\not=0 \quad \text{ on } \supp \beta_g.
\eq  
Since the action of $G$ is real analytic, we obtain a family of functions  $\beta_g(\tilde x)$ depending analytically on $g\in G(\widetilde \X)$. Multiplying the integrand of $I_\epsilon(g)$ with $\beta_g \circ \phi_\gamma(x)$, and $1-\beta_g \circ \phi_\gamma(x)$, respectively, we obtain the decomposition
\bqn 
I_\epsilon(g)= I^{(1)}_\epsilon(g)+I^{(2)}_\epsilon(g).
\eqn
Let us first examine what happens  away from the fixed points. Integrating by parts $2N$ times with respect to $\xi$ yields
\begin{gather*}
I^{(2)}_\epsilon(g)=  f(g) \sum _\gamma  \int \int _{W_\gamma}e^{i\Psi_\gamma(g,x) \cdot  \xi}  c_\gamma(x,g)(\alpha_\gamma (1-\beta_g))(\phi_{\gamma}(x)) |x_{k+1} \cdots x_{k+l}|^{s} \chi(\epsilon \xi) dx \,  \dbar \xi\\
=  f(g) \sum _\gamma  \int \int _{W_\gamma} \frac {e^{i\Psi_\gamma(g,x)\cdot  \xi}}{|\Psi_\gamma(g,x)|^{2N}} \Delta^N_\xi[\chi(\epsilon \xi)] c_\gamma(x,g)  (\alpha_\gamma (1-\beta_g))(\phi_{\gamma}(x)) |x_{k+1} \cdots x_{k+l}|^{s}  dx \,  \dbar \xi, 
\end{gather*}
 where $\Delta_\xi=\gd^2_{\xi_1}+ \dots +\gd^2_{\xi_{k+l}}$. Now,  for arbitrary $N$, 
 \bqn 
|  \Delta^N_\xi[\chi(\epsilon \xi)]| \leq C_N (1+|\xi|^2)^{-N}, 
\eqn
where $C_N$ does not depend on $\epsilon$ for $0< \epsilon \leq 1$. Furthermore, there exists a constant $M_f>0$ such that $|\Psi_\gamma(g,x)|^{2N}\geq M_f$ on the support of $1-\beta_g \circ \phi_\gamma$ for all $g \in \supp f$ and $\gamma$. By Lebesgue's theorem, we may therefore pass to the limit  under the integral, and obtain
\bqn 
\lim_{\epsilon \to 0} I^{(2)}_\epsilon(g)=0.
\eqn
Hence, as $\epsilon \to 0$, the main contributions to $I_\epsilon(g)$ originate from the fixed points of $g$. To examine these contributions, note that  condition \eqref{eq:21} implies that  $x \mapsto \phi^g_\gamma(x)-x$ defines a diffeomorphism on each of the connected components of $\supp (\alpha_\gamma\beta_g)\circ \phi_\gamma$  onto their respective  images. Performing the change of variables $y=x-\phi ^g_\gamma(x)$ we get 
\begin{align*}
I^{(1)}_\epsilon(g)&= f(g) \sum _\gamma  \int \int _{W_\gamma}e^{i\Psi_\gamma(g,x) \cdot  \xi} c_\gamma(x,g)(\alpha_\gamma \beta_g)(\phi_{\gamma}(x)) |x_{k+1} \cdots x_{k+l}|^{s} \chi(\epsilon \xi) dx \,  \dbar \xi \\
&= f(g)\sum _\gamma     \int \int  e^{-i ({\bf 1}_k\otimes T^{-1}_{x(y)})y \cdot \xi }
 |x_{k+1}(y) \cdots x_{k+l}(y)|^{s}   \frac{(\alpha_\gamma\beta_g)(\phi_{\gamma}(x(y)))c_{\gamma}(x(y),g)}{|\det (\1 -d\phi^g_\gamma(x(y)))|} \chi(\epsilon \xi) dy  \,   \dbar \xi\\
&=f(g) \sum _\gamma   \int  
 |x_{k+1}(y) \cdots
x_{k+l}(y)|^{s}  c_{\gamma}(x(y),g)  \frac{(\alpha_\gamma\beta_g)(\phi_{\gamma}(x(y)))\hat \chi(({\bf 1}_k\otimes T^{-1}_{x(y)})y/\epsilon )}{(2\pi)^{k+l}\epsilon^{k+l} |\det (\1 -d\phi^g_\gamma(x(y)))|}dy \\
&=  f(g)\sum _\gamma   \int  
 |x_{k+1}(\epsilon y) \cdots x_{k+l}(\epsilon y)|^{s}  c_{\gamma}(x(\epsilon y),g)  \frac{(\alpha_\gamma\beta_g)( \phi_{\gamma}(x(\epsilon y)))  \hat \chi(({\bf 1}_k\otimes T^{-1}_{x(\epsilon y)})y )}{(2\pi)^{k+l} |\det (\1 -d\phi^g_\gamma(x(\epsilon y)))|}dy.
 \end{align*}
Since  in a neighborhood of a  fixed point $\tilde x$ of $g$ the Jacobian of the singular change of coordinates $z=({\bf 1}_k\otimes T^{-1}_{x(\epsilon y)}) y$ converges to the expression $ |x_{k+1}(\kappa_\gamma^{-1}(\tilde x)) \cdots x_{k+l}(\kappa_\gamma^{-1}(\tilde x))|^{-1}$ as $\epsilon \to 0$, we finally obtain with $(2\pi)^{-k-l}\int \hat \chi(y ) dy=\chi(0)=1$ that
\begin{gather*}
\lim_{\epsilon \to 0} I^{(1)}_\epsilon(g)= \lim_{\epsilon \to 0} f(g)\\
\cdot   \sum _\gamma   \int  
 |x_{k+1}(\epsilon y(z)) \cdots x_{k+l}(\epsilon y(z))|^{s}  c_{\gamma}(x(\epsilon y(z)),g)  \frac{(\alpha_\gamma\beta_g)( \phi_{\gamma}(x(\epsilon y(z))))   |\gd y/\gd z| }{(2\pi)^{k+l} |\det (\1 -d\phi^g_\gamma(x(\epsilon y(z) )))|} \hat \chi(z)  dz \\
= f(g)\sum_{\tilde x \in \mathrm{Fix}(g)}  \sum _\gamma   \frac{ \alpha_\gamma (\tilde x) |x_{k+1}(\kappa_\gamma^{-1}(\tilde x)) \cdots x_{k+l}(\kappa_\gamma^{-1}(\tilde x))|^{s+1} }{ |\det (\1 -d\Phi_{g^{-1}}(\tilde x))|},
\end{gather*}
since $\bar \alpha_\gamma \equiv 1 $ on $\supp \alpha_\gamma$, and $\beta_g(\tilde x)=1$. 
The limit function $\lim_{\epsilon \to 0} I_\epsilon(g)$ is therefore clearly integrable on $G$ for $\Re s > -1$, so that by passing to the limit under the integral one computes
\begin{align*}
\Tr_s \pi(f)&=\lim_{\epsilon \to 0} I_\epsilon= \lim_{\epsilon \to 0} \int_G I_\epsilon(g) \, d_G(g)=\int _G
\lim_{\epsilon \to 0} \big (I^{(1)}_\epsilon+I^{(2)}_\epsilon\big )(g) d_G(g) \\
&=\int_G  f(g) \sum_{\tilde x \in \mathrm{Fix}(g)}  \sum _\gamma         \frac{ \alpha_\gamma (\tilde x) |x_{k+1}(\kappa_\gamma^{-1}(\tilde x)) \cdots
x_{k+l}(\kappa_\gamma^{-1}(\tilde x))|^{s+1} }{ |\det (1 -d\Phi_{g^{-1}}(\tilde x))|} \, d_G(g),
\end{align*}
 yielding the desired description of $\Theta_\pi^s$.  
 \end{proof}

As an immediate consequence of the previous theorem, we see that if $f 
\in \CT(G(\widetilde \X))$,  $\Tr_s \pi(f)$ is not singular at $s=-1$. This observation  leads to the following

\begin{corollary} 
\label{cor:2}
Let $f \in \CT(G)$ have support in $G(\widetilde \X)$. Then 
\bqn
\Tr_{reg} \pi(f)= \Tr_{-1} \pi(f)= \int_{G(\widetilde \X)}   f(g)  \sum_{\tilde x \in \mathrm{Fix}(g)}     \frac{1 }{ |\det (1 -d\Phi_{g^{-1}}(\tilde x))|} \, d_G(g).
\eqn
In particular, the distribution $\Theta_\pi:f \to \Tr_{reg}(f)$ is regular on $G(\widetilde \X)$. 
\end{corollary}
\begin{proof}
 Consider 
the Laurent expansion of $\Theta_\pi^s(f)$ at $s=-1$
given by
\bqn 
\Tr_s \pi(f)= \eklm{|x_{k+1} \cdots x_{k+l}|^s,\sum_\gamma
(\alpha_\gamma \circ \phi_{\gamma})\widetilde
A_f^\gamma(\cdot,0)}= \sum_{j=-q}^\infty
S_j\Big (\sum_\gamma
(\alpha_\gamma \circ \phi_{\gamma})\widetilde
A_f^\gamma(\cdot,0)\Big ) (s+1)^j,
\eqn
where $S_k \in \D'(\R^{k+l})$. Since by \eqref{eq:FPF}, $\Tr_s \pi(f)$ has  no pole at $s=-1$, we necessarily
must have
\bqn 
S_j\Big (\sum_\gamma
(\alpha_\gamma \circ \phi_{\gamma})\widetilde
A_f^\gamma(\cdot,0)\Big )=0 \qquad \text{ for } j<0,
\eqn
so that
\bqn 
\Tr_{-1} \pi(f)=\eklm{S_0,\sum_\gamma
(\alpha_\gamma \circ \phi_{\gamma})\widetilde
A_f^\gamma(\cdot,0)}=\Tr_{reg} \pi(f).
\eqn
The assertion now follows with the previous theorem. 
\end{proof}

In particular,  Corollary \ref{cor:2}  implies that $\Tr_{reg}\pi(f)$ is invariantly defined. Now, interpreting $\pi(g)$ as a geometric endomorphism on the trivial bundle $E=\widetilde \X \times \C$ over the Oshima compactification $\widetilde \X$, 
 a transversal trace     $\Tr^\flat \pi(g)$ of  $\pi(g) $ can be defined according to 
  \bqn 
  \Tr^\flat \pi(g)=\Tr_{\Phi_{g^{-1}}}(\Gamma(\phi_g)),
  \eqn  
  where $\phi_{g}:\Phi_{g^{-1}}^\ast E \rightarrow E$ is the associated bundle homomorphism which identifies the fiber $E_{\Phi_{g^{-1}}(\tilde x)}$ with $E_{\tilde x}$, and satisfies $(\Tr \phi_g)_{|\tilde x}=1$ at each fixed point $\tilde x$ of $g$. Taking into account  \eqref{eq:M},  the previous corollary can be reformulated, and we finally deduce the following character formula for the distribution character of $\pi$. 
\begin{theorem}
On the set of transversal elements $G(\widetilde \X)$, the distribution $\Theta_\pi:f\to\Tr_{reg}(f)$ is given by
\bqn 
\Tr_{reg} \pi(f)=\int_{G(\widetilde \X)} f(g) \Tr^\flat \pi(g) d_G(g), \qquad f \in \CT(G(\widetilde \X)),
\eqn
where 
\bqn 
\Tr^\flat \pi(g)=  \sum_{\tilde x \in \mathrm{Fix}(g)}     \frac{1 }{ |\det (1 -d\Phi_{g^{-1}}(\tilde x))|},
\eqn
the sum being over the (simple) fixed points of $g \in G(\widetilde \X)$ on $\widetilde \X$.
\end{theorem}
\qed

\section{The case $\X=\SL(3,\R)/\SO(3)$}

We shall finish this paper by describing in detail the Oshima compactification of  the Riemannian symmetric space $\X=\SL(3,\R)/\SO(3)$. Thus, let $\g= \spl(3,\R)$ be the Lie algebra of $G$. A Cartan involution $\theta:\g \rightarrow \g$ is given by $X \mapsto -X^t$, where $X^t$ denotes the transpose of $X$, and the corresponding Cartan decomposition of $\g$  reads $\g=\k\oplus \p$, where $\k=\{X \in \spl(3,\R) : X^t=-X \}$, and $\p=\{X \in \spl(3,\R): X^t=X \}$. Next, let 
\bqn
\a=\{ D(a_1,a_2,a_3): a_1,a_2,a_3 \in \R, a_1+a_2+a_3=0\},
\eqn
 where $D(a_1,a_2,a_3)$ denotes the diagonal matrix with diagonal elements $a_1,a_2$ and $a_3$. Then $\a$ is a maximal Abelian subalgebra in $\p$. Define $e_i: \a \rightarrow \R$ by $D(a_1,a_2,a_3) \mapsto a_i$, $i=1,2,3$. The set of roots $\Sigma$ of $(\g,\a)$ is given by $\Sigma=\{\pm(e_i-e_j):1 \leq i < j \leq 3\}$. We order the roots such that the positive roots are $\Sigma^+=\{e_1-e_2,e_2-e_3,e_1-e_3\}$, and obtain $\Delta=\{e_1-e_2,e_2-e_3\}$ as the set of simple roots. The root space corresponding to the root $e_1-e_2$ is given by $$\g^{e_1-e_2}=\left \{  \begin{pmatrix} 0 & x & 0 \\ 0 & 0& 0\\ 0 & 0 &0 \end{pmatrix}:x \in \R \right \},$$
and similar computations show that 
$$\g^{e_2-e_3}=\left \{  \begin{pmatrix} 0 & 0 & 0 \\ 0 & 0& z\\ 0 & 0 &0 \end{pmatrix}:z \in \R \right \},\quad \g^{e_1-e_3}=\left \{  \begin{pmatrix} 0 & 0 & y \\ 0 & 0& 0\\ 0 & 0 &0 \end{pmatrix}:y \in \R \right \}.
$$ 
For a subset $\Theta \subset \Delta$, let $\langle \Theta \rangle$ denote those elements of $\Sigma$ that are given as linear combinations of the roots in $\Theta$. Write $\langle \Theta \rangle^{\pm}$ for $\Sigma^{\pm} \cap \langle \Theta \rangle$. Put $\n^{\pm}(\Theta)=\sum_{\lambda \in \langle \Theta \rangle^{\pm}}\g^{\lambda}$,  and $\n^+_{\Theta}=\sum_{ \lambda \in \Sigma ^+ -{\langle \Theta \rangle}^+}\g^{\lambda}$. Let $\n^-_{\Theta}=\theta(\n^+_{\Theta})$. Consider now the case $\Theta = \{e_1-e_2\}$. Then  $\n^+(e_1-e_2)=\g^{e_1-e_2}$,  and $\n^+_{e_1-e_2}=\g^{e_2-e_3} \oplus \g^{e_1-e_3}$. In other words, 
 $$\n^+_{e_1-e_2}=\left \{  \begin{pmatrix} 0 & 0 & y \\ 0 & 0& z\\ 0 & 0 &0 \end{pmatrix}:y,z \in \R \right \}.$$ Exponentiating, we find that the corresponding analytic subgroups are given by 
 $$
 N^+(e_1-e_2)=\left \{  \begin{pmatrix} 1 & x & 0 \\ 0 & 1& 0\\ 0 & 0 &1 \end{pmatrix}:x \in \R \right \}, \quad N^+_{e_1-e_2}=\left \{  \begin{pmatrix} 1 & 0 & y \\ 0 & 1& z\\ 0 & 0 &1 \end{pmatrix}:y, z \in \R \right \}.$$ In a similar fashion, we obtain that
 \begin{align*}
 \n^-(e_1-e_2)=\g^{e_2-e_1}&=\left \{  \begin{pmatrix} 0 & 0 & 0 \\ x & 0& 0\\ 0 & 0 &0 \end{pmatrix}:x \in \R \right \},\\
  \n^-_{e_1-e_2}=\theta(\n^+_{e_1-e_2})&=\left \{  \begin{pmatrix} 0 & 0 & 0 \\ 0 & 0& 0\\ y & z &0 \end{pmatrix}:y,z \in \R \right \},
 \end{align*}
  and that the corresponding analytic subgroups read
  $$
  N^-(e_1-e_2)=\left \{  \begin{pmatrix} 1 & 0 & 0 \\ x & 1& 0\\ 0 & 0 &1 \end{pmatrix}:x \in \R \right \}, \quad N^-_{e_1-e_2}=\left \{  \begin{pmatrix} 1 & 0 & 0 \\ 0 & 1& 0\\ y & z &1 \end{pmatrix}:y,z \in \R \right \}.$$

The Cartan-Killing form $\langle\cdot ,\cdot \rangle: \g \times \g \rightarrow \R$  is given by $(X,Y) \mapsto \Tr(XY)$, and the modified Cartan-Killing form by $\langle X,Y\rangle_{\theta}:=-\Tr(X\theta(Y))=-\Tr(X(-Y^t))=\Tr(XY^t)$. Next, let $\a(\Theta)=\sum_{\lambda \in \langle \Theta \rangle^+}\R Q_\lambda$, where $Q_\lambda=[\theta X,X]$ for $X \in \g^{\lambda}$ such that $\langle X, X \rangle_{\theta}=1$. Also, let $\a_{\Theta}$ be the orthogonal complement of $\a(\Theta)$ with respect to $\langle\cdot ,\cdot \rangle_{\theta}$. Again, suppose that $\Theta=\{e_1-e_2\}$. We find 
$$
Q_{e_1-e_2}= \begin{pmatrix} 1 & 0 & 0 \\ 0 & -1& 0\\ 0 & 0 &0 \end{pmatrix},
$$
 so that
  $$
  \a(e_1-e_2)=\R Q_{e_1-e_2}=\left \{  \begin{pmatrix} r & 0 & 0 \\ 0 & -r& 0\\ 0 & 0 &0 \end{pmatrix}:r \in \R \right \}.
  $$
   This in turn gives us that $$\a_{e_1-e_2}=\left \{  \begin{pmatrix} a & 0 & 0 \\ 0 & a& 0\\ 0 & 0 &-2a \end{pmatrix}:a \in \R \right \}.$$ Exponentiation then shows that the corresponding analytic subgroups are $$A(e_1-e_2)=\left \{  \begin{pmatrix} a & 0 & 0 \\ 0 & a^{-1}& 0\\ 0 & 0 &1 \end{pmatrix}:a \in \R^+ \right \},A_{e_1-e_2}=\left \{  \begin{pmatrix} a & 0 & 0 \\ 0 & a& 0\\ 0 & 0 &a^{-2} \end{pmatrix}:a \in \R^+ \right \}.$$ 
Take $K=\SO(3)$ as a  maximal compact subgroup of $\SL(3,\R)$, and denote by $M_\Theta(K)$  the centralizer of $\a_\Theta$ in $K$. Observing that the adjoint action of a matrix group $G$ is just the matrix conjugation, we see that $$M_{e_1-e_2}(K)=Z_K(\a_{e_1-e_2})= \begin{pmatrix} \SO(2) & 0  \\ 0 & 1 \end{pmatrix} \cup  \begin{pmatrix} 1& 0 & 0 \\ 0 & -1& 0\\ 0 & 0 &-1 \end{pmatrix}  \begin{pmatrix} \SO(2) & 0 \\ 0 & 1  \end{pmatrix}.$$ Notice that $M_{e_1-e_2}(K)$ has 2 connected components. 
Put $M=Z_K(A)$, and let $P=MAN^+$ be the minimal parabolic subgroup given by the ordering of the roots of $(\g,\a)$. For $G=\SL(3,\R)$ one computes 
$$
M=\left \{ \begin{pmatrix} 1 & 0 & 0 \\ 0 & 1& 0 \\ 0 & 0 & 1\end{pmatrix},\begin{pmatrix} 1 & 0 & 0 \\ 0 & -1& 0 \\ 0 & 0 & -1\end{pmatrix},\begin{pmatrix} -1 & 0 & 0 \\ 0 & -1& 0 \\ 0 & 0 & 1\end{pmatrix},\begin{pmatrix} -1 & 0 & 0 \\ 0 & 1& 0 \\ 0 & 0 & -1\end{pmatrix}\right \}.
$$
 As $\Theta$ varies over the subsets of $\Delta$, we get all the parabolic subgroups $P_\Theta$ of $G$ containing $P$, and we write $P_\Theta=M_\Theta(K)AN^+$. By definition, $P_\Theta (K)=M_\Theta(K)A_\Theta N^+_\Theta$ so that, in particular, 
\begin{align*}
P_{e_1-e_2}(K)=&\left(\begin{pmatrix} \SO(2) & 0  \\ 0 & 1 \end{pmatrix} \cup  \begin{pmatrix} 1& 0 & 0 \\ 0 & -1& 0\\ 0 & 0 &-1 \end{pmatrix}  \begin{pmatrix} \SO(2) & 0 \\ 0 & 1  \end{pmatrix}\right)\cdot \left \{  \begin{pmatrix} a & 0 & 0 \\ 0 & a& 0\\ 0 & 0 &a^{-2} \end{pmatrix}:a \in \R^+ \right \}\\ &  \cdot \left \{  \begin{pmatrix} 1 & 0 & y \\ 0 & 1& z\\ 0 & 0 &1 \end{pmatrix}:z \in \R \right \}.
\end{align*}
The orbital decomposition of the Oshima compactification $\widetilde \X$ of $\X=\SL(3,\R)/ \SO(3)$ is therefore given by $$\widetilde \X=G/P  \sqcup 2\cdot G/P_{e_1-e_2}(K) \sqcup 2 \cdot G/P_{e_2-e_3}(K) \sqcup 2^2 \cdot G/K.$$

\medskip


\providecommand{\bysame}{\leavevmode\hbox to3em{\hrulefill}\thinspace}
\providecommand{\MR}{\relax\ifhmode\unskip\space\fi MR }
\providecommand{\MRhref}[2]{%
  \href{http://www.ams.org/mathscinet-getitem?mr=#1}{#2}
}
\providecommand{\href}[2]{#2}


\end{document}